\documentclass{amsart}
\usepackage{amsfonts,amssymb,amsmath,amsthm}
\usepackage{url}
\usepackage{enumerate}

\makeatletter
\@namedef{subjclassname@2010}{  \textup{2010} Mathematics Subject Classification.}
\makeatother
\newtheorem{thrm}{Theorem}[section]

\theoremstyle{definition}

\newtheorem{remark}[thrm]{Remark}
\numberwithin{equation}{section}
\email{hzoubeir2014@gmail.com}
\input{tcilatex}

\begin{document}
\address{ }
\author{Hicham Zoubeir}
\address{Ibn Tofail University, Department of Mathematics,\\
Faculty of Sciences, P. O. B : $133,$ Kenitra, Morocco.}
\title[Existence of Gevrey solutions to some polynomially nonlinear FDE]{%
Existence of Gevrey solutions to some polynomially nonlinear functional
differential equations}

\begin{abstract}
Our aim in this paper is to prove, under some growth conditions on the
datas, the solvability in a Gevrey class of a polynomially nonlinear
functional differential equation.
\end{abstract}

\dedicatory{$\emph{This}$ $\emph{modest}$ $\emph{work}$ $\emph{is}$ $\emph{%
dedicated}$ $\emph{to}$ $\emph{the}$ $\emph{memory}$ $\emph{of}$ $\emph{our}$
$\emph{beloved}$ $\emph{master}$ $\emph{Ahmed}$ $\emph{Intissar}$ $\emph{%
(1951-2017),}$ $\emph{a}$ $\emph{distinguished}$ $\emph{professor,}$ $\emph{a%
}$ $\emph{brilliant}$ $\emph{mathematician,}$ $\emph{a}$ $\emph{man}$ $\emph{%
with}$ $\emph{a}$ $\emph{golden}$ $\emph{heart.}$\emph{\ }}
\maketitle

\section{Introduction}

Functional differential equations (FDE) are differential equations in which
the derivative of the unknown function at a certain time is given in terms
of the values of the function at times which are functions of this time. FDE
are of great interest in many areas of applied sciences where processes have
an aftereffect or a delayed effect phenomena in their inner dynamics. FDE
have many applications in the theory of automatic control, the problems of
rocket motion, the problems of economical planning, the theory of population
dynamics, the study of cell cycle, the study of blood cell dynamics, the
study of infectious disease dynamics...etc. The vast literature (cf. for
example references below) concerning FDE reflects the increasingly rapid
development of this branch of mathematical analysis. The lines of research
on FDE are also various: solvability of FDE in H\"{o}lder functions,
analytic functions or entire functions, existence of periodic solutions to
FDE, existence of almost-periodic solutions to FDE, boundary value problems
for FDE or sytems of FDE, stability for FDE, theory of functional
differential inclusions,\ theory of iterative FDE, theory of stochastic FDE,
FDE with fractional derivatives, numerical approximation of solutions to
FDE, oscillatory properties of FDE, bifurcation theory of FDE.etc. Our aim
in this paper is to prove, under some regularity conditions on the datas,
the solvability in a Gevrey class $G_{k}([-1,1])$ of a nonlinear FDE of the
form $y^{\prime }(x)=a(x)P(y(\psi (x))+b(x)$ under the initial condition $%
y(d)=c$ where $a,$ $b$ and $\psi $ are holomorphic functions on some
neighborhood in $%
\mathbb{C}
$\ of $[-1,1]$, $P$ a polynomial function with real coefficients and degree
at least $2$ and $d\in \lbrack -1,1]$ and $c\in 
\mathbb{R}
$ are given numbers. The reason of our interest in Gevrey classes is firstly
their extrem usefulness in the study of problems of mathematical physics,
then the few number of published works on the solvability of FDE in these
classes of functions. Our approach relies upon basic results of functional
and complex analysis.

\section{Preliminary Notes}

Let $S$ be a nonempty subset of $%
\mathbb{C}
$ and $f:S\rightarrow 
\mathbb{C}
$ a bounded function. $\Vert f\Vert _{\infty ,S}$ denotes the quantity :%
\begin{equation*}
\Vert f\Vert _{\infty ,S}:=\underset{u\in S}{\sup }|f(u)|
\end{equation*}

For every $z\in 
\mathbb{C}
,$ we denote by $\varrho (z,S)$ the quantity :%
\begin{equation*}
\varrho (z,S):=\underset{\zeta \in S}{\inf }|z-\zeta |
\end{equation*}

For each $z\in 
\mathbb{C}
$ we denote by $\widehat{z}$ the closest point of $\left[ -1,1\right] $ to
the point $z.$

$\underline{0}_{S}$ denotes the function defined on $S$ by :%
\begin{equation*}
\underline{0}_{S}(z)=0,\text{ }z\in S
\end{equation*}

Let be $x,y\in 
\mathbb{R}
,$ we set :%
\begin{equation*}
x\vee y:=\max (x,y),\text{ }x\wedge y:=\min (x,y)
\end{equation*}

Let be $p$, $q$ $\in 
\mathbb{Z}
$, we denote by $\overline{p,q\text{ }}$ \ the set :%
\begin{equation*}
\overline{p,q\text{ }}:=\{j\in 
\mathbb{Z}
:p\leq j\leq q\}
\end{equation*}

For every function $g$ whose domain of definition contains $[-1,1]$ and
whose restriction to $[-1,1]$ is continuous on $[-1,1]$, we set :%
\begin{equation*}
||g||_{1}:=\overset{1}{\underset{-1}{\int }}|g(t)|dt
\end{equation*}

For $z\in 
\mathbb{C}
$ and $h>0$, $B(z,h)$ is the open ball in $%
\mathbb{C}
$ with center $z$ \ and radius $h$.

For $z_{1},z_{2}\in 
\mathbb{C}
$, we denote by $\underrightarrow{z_{1},z_{2}}$ the linear path joining $%
z_{1}$ to $z_{2}.$

For $r,k,A\in ]0,+\infty \lbrack $ , and $n\in 
\mathbb{N}
^{\ast }$, we set for every nonempty interval $I$ of $%
\mathbb{R}
:$%
\begin{equation*}
I_{r}:=I+B(0,r),\text{ }I_{k,A,n}:=I+B(0,An^{\frac{-1}{k}})
\end{equation*}

$C^{0}([-1,1])$ $($resp.$C^{\infty }([-1,1]))$ denotes the set of
real-valued continuous functions on $[-1,1]$ $($resp.the set of real-valued
functions of class $C^{\infty }$ on $[-1,1]).$

For $r\geq 0$ and $f\in C^{0}([-1,1])$ we denote by $\overline{\Delta }%
_{\infty }(f,r$ $)$ the closed ball in the Banach space $%
(C^{0}([-1,1]),||.||_{\infty ,[-1,1]})$ of center $f$ \ and of radius $r$.

Let $E$ be a nonempty subset of $%
\mathbb{C}
.$ By $O(E$ $)$ we denote the set of holomorphic functions on some
neighborhood of $E$.

Along this paper $k>0$ is a given real number. The Gevrey class $%
G_{k}([-1,1])$ is the set of all functions $f$ \ of class $C^{\infty }$on $%
[-1,1]$ such that there exists a constant $B>0$ such that :%
\begin{equation*}
||f^{\text{ }(n)}||_{\infty ,[-1,1]}\leq B^{n+1}n^{n(1+\frac{1}{k})},n\in 
\mathbb{N}%
\end{equation*}%
with the convention that $0^{0}=1.$

A holomorphic function $\varphi $ on a neighborhood $[-1,1]_{r}$ of $[-1,1]$
is said to verify the $E(k)$ property if there exists a constant $\tau
_{\varphi }\in ]0,r]$ depending only on $\varphi $ such that for all $A$ in $%
]0,\tau _{\varphi }]$ there exists an integer $N(A)\geq 1$ depending only on 
$A$ such that for every positive integer $p\geq N(A)$, we have :%
\begin{equation*}
\varphi ([-1,1]_{k,A,p+1})\subset \lbrack -1,1]_{k,A,p}
\end{equation*}

The real number $\tau _{\varphi }$ will be called a $k-$threshold of $%
\varphi $.

\begin{remark}
Since 
\begin{equation*}
\underset{p\text{ }\geq \text{ }N}{\cap }[-1,1]_{k,A,p+1}=[-1,1],\text{ }k>0,%
\text{ }A>0,\text{ }N\in 
\mathbb{N}
^{\ast }
\end{equation*}%
it follows that the following inclusion $:$%
\begin{equation*}
\varphi ([-1,1]\subset \lbrack -1,1]
\end{equation*}%
holds for every holomorphic function $\varphi $ on $[-1,1]$ which verifies
the $E(k)$ property.
\end{remark}

\begin{remark}
Let be $\varphi $ a holomorphic function on a neighborhood of $[-1,1]$ such
that $\varphi $ verifies the $E(k)$ property. Thence we have $:$%
\begin{equation*}
\varphi ([-1,1]_{k,A(2N(A))^{-\frac{1}{k}},p+1})\subset \lbrack
-1,1]_{k,A(2N(A))^{-\frac{1}{k}},p}\text{ },\text{ }p\in 
\mathbb{N}
^{\ast },\text{ }A\in ]0,\tau _{\varphi }[
\end{equation*}%
It follows that for every $A\in ]0,\tau _{\varphi }[$ there exists $B\in
]0,A[$ such that $:$ 
\begin{equation*}
\varphi ([-1,1]_{k,B,p+1})\subset \lbrack -1,1]_{k,B,p}\text{ },\text{ }p\in 
\mathbb{N}
^{\ast }
\end{equation*}
\end{remark}

Through this paper $a,b$ and $\psi $ are holomorphic functions on some
neighborhood $[-1,1]_{\mu }$ of $[-1,1]$ such that $\psi $ verifies the $%
E(k) $ property, and $d\in \lbrack -1,1]$ and $c\in 
\mathbb{R}
$ are given numbers . $P$ is a polynomial function with real coefficients
and degree $N_{0}\geq 2:$%
\begin{equation*}
P:x\mapsto \underset{j=0}{\overset{N_{0}}{\sum }}a_{j}x^{j}
\end{equation*}%
We associate to $P$ \ the polynomial function $\underline{P}$ defined by :%
\begin{equation*}
\underline{P}:x\mapsto \underset{j=1}{\overset{N_{0}}{\sum }}|a_{j}|x^{j}
\end{equation*}

The following \ result is a direct consequence of the fact that the
polynomial function $\underline{P}^{\text{ }\prime }$is strictly increasing
on $%
\mathbb{R}
^{+}$.

\begin{proposition}
Assume that $:$%
\begin{equation}
||a||_{1}\underline{P}^{\text{ }\prime }(0)<1  \label{(1)}
\end{equation}%
Then the equation $:$%
\begin{equation*}
||a||_{1}\underline{P}^{\text{ }\prime }(r)=1
\end{equation*}%
\ has a unique strictly positive root $\theta (P,a)$.
\end{proposition}

\section{Main Result}

Our main result is the following.

\begin{theorem}
Assume in addition to condition (\ref{(1)}) that $:$%
\begin{equation}
0<||b+P\mathit{\ }(0)a||_{1}+|c|<\theta (P,a)\mathit{\ }-\frac{\underline{P}%
^{\text{ }}(\theta (P,a))}{\underline{P}^{\text{ }\prime }(\theta (P,a))}
\label{(2)}
\end{equation}%
Then the FDE $:$%
\begin{equation}
y^{\prime }(x)=a(x)P(y(\psi (x))+b(x)  \label{(3)}
\end{equation}%
has a solution $u_{d,c}$ which belongs to the Gevrey class $G_{k}([-1,1])$\
and verifies the initial condition $:$%
\begin{equation}
y(d)=c  \label{(4)}
\end{equation}
\end{theorem}

\section{Proof of the main result}

We subdivide the proof of the main result in three steps.

\subsection{\textbf{Localisation of positive solutions of the equation}$\
(\Im )$ $r=\Vert a\Vert _{1}\protect \underline{P}^{\text{ }}(r)+\Vert
b+P(0)a\Vert _{1}+|c|$ \ }

The study of the variations on $[0,+\infty \lbrack $ of the function $%
H:t\mapsto \Vert a\Vert _{1}\underline{P}^{\text{ }}(r)+\Vert b+P(0)a\Vert
_{1}+|c|-r$ \ shows that the function $H$ \ is strictly decreasing on $%
[0,\theta (P,a)]$ and strictly increasing on $[\theta (P,a),+\infty \lbrack
. $ On the other hand we have by virtue of the condition (\ref{(2)}) :%
\begin{eqnarray*}
H(\theta (P,a)) &=&\Vert a\Vert _{1}\underline{P}^{\text{ }}(\theta
(P,a))+\Vert b+P(0)a\Vert _{1}+|c|-\theta (P,a) \\
&=&\Vert b+P(0)a\Vert _{1}+|c|-[\theta (P,a)\mathit{\ }-\frac{\underline{P}^{%
\text{ }}(\theta (P,a))}{\underline{P}^{\text{ }\prime }(\theta (P,a))}]<0
\end{eqnarray*}%
Since $H(0)>0,$ it follows that the equation $(\Im )$ has exactly in the
interval $[0,+\infty \lbrack $ two roots $r_{0}<r_{1}$ and we have :%
\begin{equation*}
0<r_{0}<\theta (P,a)<r_{1}
\end{equation*}

\subsection{\textbf{Proof \ of the existence in }$C^{\infty }([-1,1])$ 
\textbf{of a solution }$u_{d,c}$\textbf{\ of the problem (\protect \ref{(3)})
- (\protect \ref{(4)})}}

Let us consider the nonlinear operator $T$ defined in $C^{0}([-1,1])$ by the
formula :%
\begin{equation*}
T\text{ }(f\text{ })(x):=\int_{d}^{x}a(t\text{ })P(f\text{ }(\psi (t\text{ }%
)))dt+\int_{d}^{x}b\text{ }(t\text{ })dt+c\text{ },\text{ }x\in \lbrack -1,1]
\end{equation*}%
We have for all $f\in \overline{\Delta }_{\infty }(0,r_{0}):$%
\begin{eqnarray*}
\Vert T(f\text{ })\Vert _{\infty ,[-1,1]} &\leq &\Vert a\Vert _{1}\Vert
P\circ f\text{ }\circ \psi \Vert _{\infty ,[-1,1]}+\Vert b\Vert _{1}+|c\text{
}| \\
&\leq &\Vert a\Vert _{1}\Vert \underline{P}\circ f\text{ }\circ \psi \Vert
_{\infty ,[-1,1]}+\Vert b\Vert _{1}+|c\text{ }| \\
&\leq &\Vert a\Vert _{1}\underline{P}\text{ }(r_{0})+\Vert b\Vert _{1}+|c%
\text{ }|=r_{0}
\end{eqnarray*}%
Thence the closed ball $\overline{\Delta }_{\infty }(0,r_{0})$ is stable by
the operator $T$. On the other hand according to the remark 2.1, we have,
for all $u$ $\in $ $\overline{\Delta }_{\infty }(0,r_{0})$ and $x_{1},$ $%
x_{2}\in \lbrack -1,1]:$%
\begin{eqnarray*}
||T\text{ }(v)-T\text{ }(u)||_{\infty ,[-1,1]} &\leq &(||a\text{ }||_{1}||%
\text{ }(P\circ v\circ \psi )-(P\circ u\circ \psi )||_{\infty ,[-1,1]} \\
&\leq &(||a\text{ }||_{1}||\text{ }P^{\prime }||_{\infty
,[-r_{0},r_{0}])}||(v\circ \psi )-(u\circ \psi )||_{\infty ,[-1,1]} \\
&\leq &||a\text{ }||_{1}\text{ }\underline{P}^{\prime
}(r_{0})||v-u||_{\infty ,[-1,1]}
\end{eqnarray*}%
Since $0<r_{0}<\theta (P)$ it follows that $0<||a$ $||_{1}$ $\underline{P}%
^{\prime }(r_{0})<||a$ $||_{1}$ $\underline{P}^{\prime }(\theta (P))=1.$
Thence the operator $T$ $|_{\overline{\Delta }_{\infty }(0,r_{0})}$ is
lipshitzian with Lipshitz constant $\Vert a\Vert _{1}\underline{P}^{\text{ }%
\prime }(r_{0})\in \lbrack 0,1[.$ It follows that the operator $T$ \ has\ a
unique fixed point $u_{d,c}$ $\in $ $\overline{\Delta }_{\infty }(r_{0}).$
Thence $u_{d,c}$ is a solution in $C^{\infty }([-1,1])$ of the problem (\ref%
{(3)}) - (\ref{(4)}).

\subsection{\textbf{Proof that }$u_{d,c}$\textbf{\ belongs to the Gevrey
class }$G_{k}([-1,1])$}

Let us consider the sequence of functions $(f_{n})_{n\in \text{ }%
\mathbb{N}
^{\ast }}$ defined on $[-1,1]$ by the following relations :%
\begin{equation*}
f_{1}:=\underline{0}_{[-1,1]},\text{ }f_{n+1}=T(f_{n})
\end{equation*}%
This sequence is well defined and we have :%
\begin{equation*}
f_{n}\in \overline{\Delta }_{\infty }(r_{0}),\text{ }n\in 
\mathbb{N}
^{\ast }
\end{equation*}%
Furthermore $(f_{n})_{n\in \text{ }%
\mathbb{N}
^{\ast }}$ is uniformly convergent on $[-1,1]$ to $u_{d,c}$. On the other
hand let us consider the sequence of functions $(F_{n})_{n\in 
\mathbb{N}
^{\ast }}$of functions, $F_{n}:[-1,1]_{k,\nu ,n}\rightarrow 
\mathbb{C}
$ defined by the following relations :%
\begin{eqnarray*}
F_{1} &:&=\underline{0}_{[-1,1]_{k,,\nu ,1}} \\
F_{n+1}(z) &:&=c+\underset{\underrightarrow{d,z}}{\int }\ a(\zeta
)P(F_{n}(\psi (\zeta ))d\zeta +\underset{\underrightarrow{d,z}}{\int }%
b(\zeta )d\zeta ,\text{ }z\in \ [-1,1]_{k,\nu ,n+1},n\in 
\mathbb{N}
^{\ast }
\end{eqnarray*}%
where $\nu \in ]0,\mu \wedge \tau _{\psi }[$ is choosen, as in remark 2.1,
such that :%
\begin{equation*}
\psi ([-1,1]_{k,\nu ,n+1})\subset \lbrack -1,1]_{k,\nu ,n},\text{ }n\in 
\mathbb{N}
^{\ast }
\end{equation*}%
Direct computations show that the function $F_{n}$ is for every $n\in 
\mathbb{N}
^{\ast }$ well defined and holomorphic on $[-1,1]_{k,\nu ,n}$. Furthermore $%
F_{n}$ is for every $n\in 
\mathbb{N}
^{\ast }$ an extension to $[-1,1]_{k,\nu ,n}$ of the function $f_{n}.$ We
associate to every number $s\in \lbrack 0,\frac{\nu }{2}]$ the recurrent
numerical sequence $\omega (s):=(\omega _{n,s})_{n\geq 1}$ defined by the
relations :%
\begin{eqnarray*}
\omega _{1,s} &:&=1 \\
\omega _{n+1,s} &:&=||a||_{\infty ,[-1,1]_{\frac{\mu }{2}}}\underline{P}%
(r_{0}+sn^{\frac{-1}{k}}\omega _{n,s})+||b+P(0)a||_{\infty ,[-1,1]_{\frac{%
\nu }{2}}},\text{ }n\in 
\mathbb{N}
^{\ast }
\end{eqnarray*}%
Direct computations show the existence of a constant $C>(\frac{2}{\nu })\vee
1$ which depends only on $(a,b,P)$ and satisfies the following estimate for
every $s\in \lbrack 0,\frac{\nu }{2}]$ and $n\geq 1:$%
\begin{equation*}
||a||_{\infty ,[-1,1]_{\frac{\mu }{2}}}\underline{P}(r_{0}+sn^{\frac{-1}{k}%
}\omega _{n,s})+||b+P(0)a||_{\infty ,[-1,1]_{\frac{\mu }{2}}}]\leq C[(sn^{%
\frac{-1}{k}}\omega _{n,s})\vee 1]^{N_{0}}
\end{equation*}%
We observe that $\frac{1}{C}\in ]0,(\frac{\nu }{2})\wedge \tau _{\psi }[.$
An easy induction on $n\geq 1$ shows then that we have for all $s\in ]0,%
\frac{1}{C}]:$%
\begin{equation}
0<\omega _{n,s}\leq C\text{ },\text{ }n\in 
\mathbb{N}
^{\ast }  \label{(5)}
\end{equation}%
On the other hand according to the remark 2.2, there exists a real $s_{1}\in
]0,\frac{1}{C}[$ such that :%
\begin{equation}
\psi ([-1,1]_{k,s_{1},n+1})\subset \lbrack -1,1]_{k,s_{1},n}\text{ },\text{ }%
n\in 
\mathbb{N}
^{\ast }  \label{(6)}
\end{equation}

\begin{proposition}
For such a choice of the number $s_{1}$, we have for every $n\in 
\mathbb{N}
^{\ast }:$%
\begin{equation*}
F_{n}([-1,1]_{k,s_{1},n})\subset \lbrack -r_{0},r_{0}]_{_{k,Cs_{1},n}}
\end{equation*}
\end{proposition}

\begin{proof}
Let us prove by induction that for all $n\in 
\mathbb{N}
^{\ast }$ the following inclusion holds%
\begin{equation*}
F_{n}([-1,1]_{k,s_{1},n})\subset \lbrack -r_{0},r_{0}]_{\omega
_{n,s_{1}}s_{1}n^{\frac{-1}{k}}}
\end{equation*}%
We denote the last inclusion by $\tciFourier (n).$ We have obviously :%
\begin{equation*}
F_{1}([-1,1]_{k,s_{1},n})\subset \lbrack -r_{0},r_{0}]_{\omega
_{1,s_{1}}s_{1}}
\end{equation*}%
Thence $\tciFourier (1)$ is true. Let us assume for a certain $n\in 
\mathbb{N}
^{\ast }$ that the inclusions $\tciFourier (p)$ are true for all $p\in 
\overline{1,n\text{ }}.$ We have for each $z\in \lbrack -1,1]_{k,s_{1},n+1}:$
\begin{eqnarray*}
&&\varrho (F_{n+1}(z),[-r_{0},r_{0}]) \\
&\leq &|F_{n+1}(z)-F_{n+1}(\widehat{z})| \\
&\leq &\left \vert \underset{\underrightarrow{\widehat{z},z}}{\int }\
a(\zeta )(P(F_{n}(\psi (\zeta ))-P(0))d\zeta \right \vert + \\
&&+\left \vert \underset{\underrightarrow{\widehat{z},z}}{\int }(b(\zeta
)+P(0)a(\zeta )d\zeta \right \vert \\
&\leq &\left( \underset{\underrightarrow{\widehat{z},z}}{\int }|a(\zeta
)|.|d\zeta |\right) ||P-P(0)||_{\infty ,F_{n}(\psi ([-1,1]_{k,s_{1},n+1}))}+
\\
&&+\underset{\underrightarrow{\widehat{z},z}}{\int }|b(\zeta )+P(0)a(\zeta
)|.|d\zeta |
\end{eqnarray*}%
It follows from (\ref{(6)}) that%
\begin{eqnarray*}
&&\varrho (F_{n+1}(z),[-r_{0},r_{0}]) \\
&\leq &\left( \underset{\underrightarrow{\widehat{z},z}}{\int }|a(\zeta
)||d\zeta |\right) ||P-P(0)||_{\infty ,F_{n}([-1,1]_{k,s_{1},n})}+ \\
&&+\underset{\underrightarrow{\widehat{z},z}}{\int }|b(\zeta )+P(0)a(\zeta
)||d\zeta | \\
&\leq &\left[ 
\begin{array}{c}
||a||_{\infty ,[-1,1]_{s}}\underline{P}(r_{0}+\omega _{n,s_{1}}s_{1}n^{\frac{%
-1}{k}})+ \\ 
+||b+P(0)a||_{\infty ,[-1,1]_{s}}%
\end{array}%
\right] \cdot \\
&&\cdot s_{1}(n+1)^{\frac{-1}{k}} \\
&\leq &\omega _{n+1,s_{1}}s_{1}(n+1)^{\frac{-1}{k}}
\end{eqnarray*}%
Thence $\tciFourier (n+1)$ is true. Consequently $\tciFourier (n)$ is true
for all $n\in 
\mathbb{N}
^{\ast }.$It follows from (\ref{(5)}) that we have for all $n\in 
\mathbb{N}
^{\ast }:$%
\begin{equation*}
\  \ F_{n}([-1,1]_{k,s_{1},n})\subset \lbrack -r_{0},r_{0}]_{_{k,Cs_{1},n}}
\end{equation*}%
The proof of the proposition is then complete.
\end{proof}

End of te proof of the main result :

The function $\Lambda $ defined on $[0,\mu \lbrack $\  \ by the relations :%
\begin{eqnarray*}
\Lambda (s) &:&=\underset{z\in \lbrack |1,1]_{s}}{\sup }\left( \underset{%
\underrightarrow{d,z}}{\int }|a(\zeta )|.|d\zeta |\right) .\underline{P}^{%
\text{ }\prime }(r_{0}+Cs),\text{ }s\in ]0,\mu \lbrack \\
\Lambda (0) &:&=\underset{x\in \lbrack -1,1]}{\sup }\left( \underset{d\wedge
x}{\overset{d\vee x}{\int }}|a(t)|dt\right) .\underline{P}^{\text{ }\prime
}(r_{0})
\end{eqnarray*}%
is continuous on $[0,\mu \lbrack $ and we have $\Lambda (0)\in \lbrack 0,1[.$
Thence, according to proposition 3, there exists a real $s_{2}\in
]0,s_{1}\wedge (\frac{\nu }{2C})[$ such that :%
\begin{equation}
\Lambda (Cs_{2})<1  \label{(7)}
\end{equation}%
\begin{equation}
F_{n}([-1,1]_{k,s_{2},n})\subset \lbrack -r_{0},r_{0}]_{_{k,Cs_{2},n}},\text{
}n\in 
\mathbb{N}
^{\ast }  \label{(8)}
\end{equation}%
It follows from (\ref{(8)}) that we have for all $n\in 
\mathbb{N}
^{\ast }\backslash \{1\}$ and $z\in \lbrack -1,1]_{k,s_{2},n+1}:$%
\begin{eqnarray*}
&&|F_{n+1}(z)-F_{n}(z)| \\
&\leq &\underset{\underrightarrow{d,z}}{\int }|a(\zeta )|.|P(F_{n}(\psi
(\zeta ))-P(F_{n-1}(\psi (\zeta ))|.|d\zeta |) \\
&\leq &\left( \underset{\underrightarrow{d,z}}{\int }|a(\zeta )|.|d\zeta
|\right) ||P\text{ }^{\prime }||_{\infty
,J_{k,Cs_{2},n}}||F_{n}-F_{n-1}||[-1,1]_{k,s_{2},n+1} \\
&\leq &\left( \underset{\underrightarrow{d,z}}{\int }|a(\zeta )|.|d\zeta
|\right) )\underline{P^{\prime }}%
(r_{0}+Cs_{2})||F_{n}-F_{n-1}||_{[-1,1]_{k,s_{2},n}} \\
&\leq &\Lambda (Cs_{2})||F_{n}-F_{n-1}||_{\infty ,[-1,1]_{k,s_{2},n}}
\end{eqnarray*}%
It follows that :%
\begin{equation*}
||F_{n+1}-F_{n}||_{\infty ,[-1,1]_{k,s_{2},n+1}}\leq \Lambda
(Cs_{2})||F_{n}-F_{n-1}||_{\infty ,[-1,1]_{k,s_{2},n}}\text{ },\text{ }n\in 
\mathbb{N}
^{\ast }\backslash \{1\}
\end{equation*}%
Thence we have : 
\begin{equation*}
||F_{n+1}-F_{n}||_{\infty ,[-1,1]_{k,s_{2},n+1}}\leq (\Lambda (Cs_{2}))^{n}%
\frac{||F_{n+1}-F_{n}||_{\infty ,[-1,1]_{k,s_{2},2}}}{\Lambda (Cs_{2})}\text{
},\text{ }n\in 
\mathbb{N}
^{\ast }
\end{equation*}%
\ Let us set : 
\begin{equation*}
\left \{ 
\begin{array}{c}
g_{0}:=F_{1} \\ 
g_{n}:=F_{n+1}-F_{n}\text{ , }n\in 
\mathbb{N}
^{\ast } \\ 
C_{0}:=\frac{||F_{2}-F_{1}||_{\infty ,[-1,1]_{k,s_{2},2}}}{\Lambda \ (Cs_{2})%
} \\ 
\delta :=\Lambda \ (Cs_{2})%
\end{array}%
\right.
\end{equation*}%
Then we have for all $n\in 
\mathbb{N}
:$%
\begin{equation*}
\left \{ 
\begin{array}{c}
g_{n}\in O\text{ }([-1,1]_{k,s_{2},n+1}) \\ 
g_{n}|_{[-1,1]}=f_{n+1}-f_{n}\text{ , }n\in 
\mathbb{N}
^{\ast } \\ 
||g_{n}||_{_{\infty ,[-1,1]_{k,s_{2},n+1}}}\leq C_{0}\text{ }\delta
^{n},n\in 
\mathbb{N}
^{\ast }%
\end{array}%
\right.
\end{equation*}%
On the other hand in view of Cauchy's in\'{e}qualities we have the following
estimates : \ 
\begin{equation*}
||g_{n}^{(p)}|_{[-1,1]}||_{\infty ,[-1,1]}\leq C_{0}p!\left( \frac{s_{2}n^{%
\frac{-1}{k}}}{2}\right) ^{-p}\delta ^{n},\text{\  \ }p\in 
\mathbb{N}
,\text{ }n\in 
\mathbb{N}
^{\ast }
\end{equation*}%
It follows that :%
\begin{equation}
||g_{n}^{(p)}|_{[-1,1]}||_{\infty ,[-1,1]}\leq C_{0}(\frac{2}{s_{2}})^{p}%
\underset{x\text{ }\geq 0}{\sup }(\sqrt{\delta }^{x}x^{\frac{p}{k}}).\sqrt{%
\delta }^{n}p^{p},\text{ \ }p\in 
\mathbb{N}
,\text{ }n\in 
\mathbb{N}
^{\ast }  \label{(9)}
\end{equation}%
Thanks to (\ref{(7)}) we have\ $\delta \in \lbrack 0,1[$. Thence we can show
by directs computations that : 
\begin{equation*}
\underset{x\text{ }\geq 0}{\sup }(\sqrt{\delta }^{x}x^{\frac{p}{k}})=\left( 
\frac{2^{\frac{1}{k}}}{(ek\ln (\frac{1}{\delta }))^{\frac{1}{k}}}\right)
^{p}p^{\frac{p}{k}},\text{ }p\in 
\mathbb{N}%
\end{equation*}%
The estimate (\ref{(9)}) becomes then :%
\begin{equation}
||g_{n}^{(p)}|_{[-1,1]}||_{\infty ,[-1,1]}\leq C_{0}\left( \frac{2^{\frac{1+k%
}{k}}}{s_{2}(ek\ln (\frac{1}{\delta }))^{\frac{1}{k}}}\right) ^{p}(\sqrt{%
\delta })^{n}p^{p(1+\frac{1}{k})},\text{ \ }p\in 
\mathbb{N}
,\text{ }n\in 
\mathbb{N}
^{\ast }  \label{(10)}
\end{equation}%
\ Consequently the function series $\sum g_{n}^{(p)}|_{[-1,1]}$ is for every 
$p\in 
\mathbb{N}
$ uniformly convergent on $[-1,1].$We already knew that\ the function series 
$\sum g_{n}|_{[-1,1]}$ is uniformly convergent on $[-1,1]$ to the function $%
u_{d,c}$.\ It follows\ then that :%
\begin{equation*}
\  \Vert (u_{d,c}{}-f_{1})^{(p)}\Vert _{\infty ,[-1,1]}\leq \frac{C_{0}\sqrt{%
\delta }}{1-\sqrt{\delta }}\left( \frac{2^{\frac{1+k}{k}}}{s_{2}(ek\ln (%
\frac{1}{\delta }))^{\frac{1}{k}}}\right) ^{p}p^{p(1+\frac{1}{k})},\text{ \ }%
p\in 
\mathbb{N}
\ 
\end{equation*}%
Since $f_{1}$ is real-analytic on $[-1,1]$, it follows that $u_{d,c}$
belongs to the Gevrey class $G_{k}([-1,1]).$

The proof of the main result is then complete.

\section{Examples}

\begin{proposition}
The function $\sin $ verifies the $E(1)$ property.
\end{proposition}

\begin{proof}
Let $A\in ]0,1[$, $p\in 
\mathbb{N}
^{\ast }$and $z\in \lbrack -1,1]_{1,p+1,A}.$ We have the following
inequalities :%
\begin{eqnarray*}
\varrho (\sin z,[-1,1]) &\leq &|\sin z-\sin \widehat{z}| \\
&\leq &\overset{+\infty }{\underset{j=1}{\sum }}\frac{|\sin ^{(j)}(\widehat{z%
}))|}{j!}|z-\widehat{z}|^{j} \\
&\leq &\overset{+\infty }{\underset{j=1}{\sum }}\varrho (z,[-1,1])^{j} \\
&\leq &\frac{A}{p+1-A} \\
&<&\frac{A}{p}
\end{eqnarray*}%
It follows that :%
\begin{equation*}
\sin ([-1,1]_{1,p+1,A})\subset \lbrack -1,1]_{1,p,A}\text{ },p\in 
\mathbb{N}
^{\ast }
\end{equation*}%
Since the function $\sin $ is entire it follows that $\sin $ verifies the $%
E(1)$ property.
\end{proof}

\begin{example}
\textbf{\ }
\end{example}

Let $\alpha \in 
\mathbb{R}
^{\ast },\beta >0,\gamma \in 
\mathbb{R}
^{\ast }$ and $N\in 
\mathbb{N}
$ be real constants. We assume that :%
\begin{equation}
\frac{48\beta ^{2}|\alpha |}{\gamma ^{2}}<N+1  \label{(11)}
\end{equation}%
Let us consider the FDE :%
\begin{equation}
y^{\prime }(t)=\alpha t^{N}(y(\sin t))^{3}+\beta \cosh \gamma t  \label{(12)}
\end{equation}%
with the initial condition :%
\begin{equation}
y(0)=0  \label{(13)}
\end{equation}%
Consider then the functions :%
\begin{equation*}
P_{1}:z\mapsto z^{3},a_{1}:z\mapsto \alpha z^{N},b_{1}:z\mapsto \beta \cosh
\gamma z
\end{equation*}%
We have :%
\begin{eqnarray*}
\theta (P_{1},a_{1}) &=&\sqrt{\frac{N+1}{6|\alpha |}} \\
\theta (P_{1},a_{1})\mathit{\ }-\frac{\underline{P}^{\text{ }}(\theta
(P_{1},a_{1}))}{\underline{P}^{\text{ }\prime }(\theta (P_{1},a_{1}))} &=&%
\sqrt{\frac{N+1}{24|\alpha |}}
\end{eqnarray*}%
We have also :%
\begin{equation*}
||b_{1}+P\mathit{\ }(0)a_{1}||_{1}+|0|=\frac{2\beta }{\gamma }\sinh \gamma <%
\sqrt{\frac{N+1}{24|\alpha |}}
\end{equation*}%
It follows that :%
\begin{equation*}
||b_{1}+P\mathit{\ }(0)a_{1}||_{1}+|0|<\theta (P_{1},a_{1})\mathit{\ }-\frac{%
\underline{P}^{\text{ }}(\theta (P_{1},a_{1}))}{\underline{P}^{\text{ }%
\prime }(\theta (P_{1},a_{1}))}
\end{equation*}%
Thence according to the proposition $5.1$ and to the main result, it follows
that the problem (\ref{(12)}) - (\ref{(13)}) has a solution which belongs to
the Gevrey classe $G_{1}([-1,1]).$

\begin{example}
\end{example}

Consider the FDE :%
\begin{equation}
y^{\prime }(t)=2\ln 2.2^{t}(y(\sin t)^{4}-y(\sin t)^{2}+\frac{1}{8}y(\sin
t)-1)+\frac{9\ln 2}{500}.2^{t}  \label{(14)}
\end{equation}%
with the initial condition :%
\begin{equation}
\text{\ }y(0)=\frac{1}{100}  \label{(15)}
\end{equation}%
Consider the functions :%
\begin{equation*}
P_{2}:z\mapsto z^{4}-z^{2}+\frac{1}{8}z-1,\text{ }a_{2}:z\mapsto 2\ln
2.2^{z},\text{ }b_{2}:z\mapsto \frac{301\ln 2}{150}2^{z}
\end{equation*}%
We have :%
\begin{equation*}
||a||_{1}\underline{P_{2}^{\prime }}(0)=\frac{3}{8}<1
\end{equation*}%
Thence the equation $||a||_{1}\underline{P_{2}^{\prime }}(t)=1$ has a unique
positive root $\theta (P_{2},a_{2}).$Moreover, the computer algebra system
Maple provides the following continued inequalities :%
\begin{eqnarray*}
0,1020416497 &<&\theta (P_{2},a_{2})<0,1020416498 \\
0,0289635672 &<&\theta (P_{2},a_{2})\mathit{\ }-\frac{\underline{P}^{\text{ }%
}(\theta (P_{2},a_{2}))}{\underline{P}^{\text{ }\prime }(\theta
(P_{2},a_{2}))}<0,0289635673
\end{eqnarray*}%
We have also :%
\begin{equation*}
\left \Vert b_{2}+P_{2}(0)a_{2}\right \Vert _{1}+\left \vert \frac{1}{100}%
\right \vert =0,02
\end{equation*}%
Thence : 
\begin{equation*}
\left \Vert b_{2}+P_{2}(0)a_{2}\right \Vert _{1}+\left \vert \frac{1}{100}%
\right \vert <\theta (P_{2},a_{2})\mathit{\ }-\frac{\underline{P}^{\text{ }%
}(\theta (P_{2},a_{2}))}{\underline{P}^{\text{ }\prime }(\theta
(P_{2},a_{2}))}
\end{equation*}%
It follows, according to the proposition $5.1$ and to the main result, that
the problem (\ref{(14)}) - (\ref{(15)}) has a solution which belongs to the
Gevrey class $G_{1}([-1,1]).$

\bigskip

\end{document}